\documentclass{article}

\usepackage{amsmath,amsthm,amsfonts,amssymb}
\usepackage{tikz-cd}
\usepackage{hyperref}
\usepackage[abbrev,nobysame]{amsrefs}

\usepackage[english]{babel}
\usepackage[T1]{fontenc}
\usepackage[utf8]{inputenc}

\newtheorem{corollary}{Corollary}[section]

\newtheorem{proposition}{Proposition}[section]

\theoremstyle{definition}

\theoremstyle{definition}

\theoremstyle{definition}

\theoremstyle{definition}
\newtheorem{remark}{Remark}[section]

\renewcommand{\Im}{\operatorname{Im}}

\DeclareMathOperator{\Tr}{Tr}

\title{Generalized theta functions, projectively flat vector bundles and noncommutative tori}
\author{Maximiliano SANDOVAL${}^{\dagger}$ and Mauro SPERA${}^{\dagger\dagger}$\\
\phantom{void}\\
${}^{\dagger}$ Departamento de Mat\'ematica \\ 
Pontificia Universidad Cat\'olica de Chile\\ 
Campus San Joaqu\'\i n. 
Avenida Vicu\~na Mackenna 4860, Macul, Chile\\
\phantom{void}\\
 ${}^{\dagger\dagger}$ Dipartimento di Matematica e Fisica  ``Niccol\`o Tartaglia''\\
Universit\`a Cattolica del Sacro Cuore\\
Via della Garzetta 48, 25133 Brescia, Italia\\
\begin{small}
  \texttt{msandova@protonmail.com},
  \texttt{mauro.spera@unicatt.it}
\end{small}
}

\begin{document}

\maketitle
\begin{abstract}
 In this paper, the well-known relationship between theta functions and Heisenberg group actions thereon is resumed by combining complex algebraic and noncommutative geometric techniques in that we describe Hermitian-Einstein vector bundles on 2-tori via representations of noncommutative tori, thereby reconstructing Matsushima's setup~\cite{Matsushima}
and elucidating the ensuing Fourier-Mukai-Nahm (FMN) aspects. We prove the existence of noncommutative torus actions on the space of smooth sections of  Hermitian-Einstein vector bundles on 2-tori preserving the eigenspaces of a natural Laplace operator. Motivated by the Coherent State Transform approach to theta functions~(\cite{FMN, Tyurin}), we extend the latter to vector valued thetas and develop an additional algebraic reinterpretation of Matsushima's theory  making FMN-duality manifest again.
\end{abstract}
\medskip

\begin{footnotesize}
  \par

  \smallskip
  \noindent
  \textbf{MCS 2020}: 14 K 25; 32 L 05; 14 F 06; 58 B 34.
\end{footnotesize}

\section{Introduction}

In this paper we address, in the simplest case, the well-known intriguing and multifaceted relationship between theta functions and representations of Heisenberg groups (both infinite and finite~\cite{Matsushima,Mumford, Polishchuck}), from a blended complex differential geometric viewpoint -- focussed on holomorphic vector bundles on 2-tori -- and a noncommutative geometric one -- involving (rational) noncommutative 2-tori -- possibly bringing in some novel insights and, in particular, improving the treatment given in~\cite{Spera2015}. Noncommutative geometry arose with the aim of studying singular objects, such as orbit spaces, generally intractable via traditional topological, analytical and geometrical tools (see in particular the comprehensive \cite{ConnesNCG}), and it is ultimately based on the transition from points in a topological space to functions thereon and thence to general algebras. Noncommutative tori provide a simple yet highly non trivial testing ground for carrying out such a programme.  They appear naturally
in condensed matter physics issues, see e.g.~\cite{Bellissard-et-al, Denittis} and they also implicitly crop up in the theory of projectively flat vector bundles over tori, see e.g.~\cite{Matsushima,Kobayashi}; it is precisely this aspect that is dealt with in the present work.\par
Specifically, in Section 3, via a series of Propositions, we prove the existence of a representation of the noncommutative torus
${\mathcal A}_{1/\theta}$ ($\theta = q/r$, $q$ and $r$ being coprime positive integers)
on the space of sections $\Gamma({\mathcal E}_{r,q})$ of a projectively flat
\emph{Hermitian-Einstein} holomorphic vector bundle (or \emph{HE}-vector bundle
for short) ${\mathcal E}_{r,q}$ of rank $r$ and degree $q$ on a two-dimensional torus. This representation  will actually preserve the eigenspaces (``Landau levels'') of a natural Laplace operator (essentially, a quantum harmonic oscillator), hence, in particular, its holomorphic sections, thereby recovering the classical algebraic-geometric portrait.
The vector bundle ${\mathcal E}_{r,q}$ itself, in turn, can be manufactured from
a representation of ${\mathcal A}_{\theta}$ on its typical fibre. Another
representation of ${\mathcal A}_{\theta}$ on $\Gamma({\mathcal E}_{r,q})$, commuting with
the representation of ${\mathcal A}_{1/\theta}$, is produced out of the parallel transport pertaining to the Chern connection on ${\mathcal E}_{r,q}$.
The above developments bring in a vivid portrait of the Fourier-Mukai-Nahm (FMN,~\cite{Mukai,Nahm}) duality between ${\mathcal E}_{r,q}$ and ${\mathcal E}_{q,r}$ together with their respective Chern connections. Actually, all objects, representations and bundles, will come in (torus-) families (moduli).
In Section 4,  upon resorting to the well-known heat equation interpretation
of theta functions (described via the so-called \emph{Coherent State Transform (CST)} of~\cite{FMN}) and
further insisting on a noncommutative torus perspective, we present a ``matrix'' description of Matsushima's theory  making again the above duality manifest. Finally, we prove that, as pre-$C^*$-algebras (and for the unique $C^*$-tensor product involved), ${\mathcal A}_{q/r}\otimes{\mathcal A}_{r/q}$
and ${\mathcal A}_{1/rq}$ are isomorphic. This will be a byproduct of a ``categorical'' reinterpretation of Gauss sums identities also shedding light on Fourier-Mukai-Nahm transform issues. Moreover, a vector analogue of the $\emph{CST}$ will be set up. The layout of the paper
is completed by Section 2 -- gathering together background
material from different areas in order to fix notation and to pave the ground for
the successive developments in Sections 3 and 4 -- and by Section 5, pointing out possible applications and further research directions.

\section{Preliminary tools}

In this section we establish our notation and collect several miscellaneous technical
tools for the benefit of a wider readership.

\subsection{$k$-level theta functions and the Coherent State Transform}

We begin by providing minimal background on $k$-level theta functions and on their relationship with the \emph{heat equation} closely following the exposition of~\cite{FMN} (see also~\cite{Tyurin})---up to slight notational changes---and referring to it for a complete treatment. We restrict to the genus one case, namely to an Abelian torus $(M, \tau)$,
$\tau \in {\mathbb C}$, $\Im \tau > 0$.
\par
Let us start from the following (tempered) distributions on $S^1$
$$
\theta^0_{\ell}(x) = \sum_{n\in {\mathbb Z}} e^{2\pi i(\ell + kn)x}
$$
with $\ell = 0,1, \dots, k-1$. They are mapped, via the so-called \emph{Coherent State Transform (CST)}:
\[
  CST (\theta^0_{\ell})(z) = \vartheta_{\ell} (z, \tau) = \sum_{n\in {\mathbb Z}} e^{\pi i (\ell + kn) (\tau/k)  (\ell + kn)} e^{2\pi i(\ell + kn)z}
\]
to the $k$-\emph{level theta functions}. These, in turn, are interpreted as
holomorphic sections of the $k$-th power of the so called \emph{Theta line
  bundle}, and yield a basis thereof, as a consequence of the Riemann-Roch
theorem. A far reaching generalization for HE-vector
bundles has been developed by Matsushima~\cite{Matsushima} and his theory will
be retrieved and elaborated on in what follows.

\subsection{Review of Matsushima's theory}

In this subsection we outline Matsushima's theory~\cite{Matsushima}, tailoring the exposition to our purposes and referring to~\cite{Kobayashi}, especially Ch.IV-7 and to~\cite{Spera2015}, Section 3.2, for background material.
Here we just recall that an irreducible holomorphic vector bundle (i.e.\ without proper
holomorphic direct summands) admits a HE-metric if and only if it is stable in
the algebraic-geometric sense: this is the celebrated \emph{Kobayashi-Hitchin} correspondence, fully established for compact K\"ahler manifolds in~\cite{UY}. HE-vector bundles are poly-stable,
i.e. direct sums of stable bundles. In the present work we shall consider the special class of HE-bundles consisting of projectively flat holomorphic vector bundles on complex tori, which are equipped with a Hermitian metric whose corresponding canonical (Chern) connection has constant curvature. \par
\smallskip
Let $r$ and $q$ be coprime positive integers, i.e.\ $\gcd(r,q) = 1$.
Let $V$ be a one-dimensional complex vector space and let us consider a complex
torus $V/L$ where $L \cong {\mathbb Z}^2$ is a lattice. Let
$L^{\prime} \subset L$ be a  complete sublattice of $L$.
Specifically, if $L = \langle \omega_1, \omega_2\rangle$ is the lattice
generated by a (real) basis $\{\omega_j\}_{j=1,2}$ of $V$,
let $L^{\prime} = \langle r \omega_1, \omega_2\rangle$ and
$$
K := L/ L^{\prime} \cong {\mathbb Z}_r,
$$
thus we have an $r$-covering of complex tori
$$
\varphi: V/L^{\prime} \longrightarrow V/L .
$$
Let $A$ be the ${\mathbb Q}$-valued form defined by
 $A(\omega_1, \omega_2) = q/r$, and $A^{\prime} = rA$ the  ${\mathbb Z}$-valued form
 fulfilling $A^{\prime}(\omega_1, \omega_2) = q$ (the pull-back of $A$ via $\varphi$). The form $A$ gives rise to a HE-vector bundle
${\mathcal E}_{r,q} \to V/L$ ---of rank $r$ and degree $q$--- i.e.\ such that its canonical (Chern) connection has constant curvature

$\Omega = -2\pi i A \otimes {\rm Id}_{{\mathcal E}_{r,q}}$. Correspondingly, one has a HE-line bundle
${\mathcal E}_{1,q} \to V/L^{\prime}$, the $q$-level theta line bundle over $V/L^{\prime}$, related to the form $A^{\prime}$. Let us denote, as usual, by $H^0(X,E)$ the space of holomorphic sections of a holomorphic vector bundle $E \to X$, $X$ being the base manifold, with dimension $h^0(X,E)$. It is known (by Riemann-Roch and the vanishing of $H^{1}(X,E)$ for tori) that 
$$
h^0({V/L,\mathcal E}_{r,q}) = h^0(V/L^{\prime},{\mathcal E}_{1,q}) = q,
$$ 
thus the corresponding section spaces are (non-canonically) isomorphic. Now, given an orthonormal basis of  $H^0(V/L^{\prime},{\mathcal E}_{1,q})$ made up by $q$-level theta functions $\{ \vartheta_m\}_{m=0}^{q-1}$, one has, according to Matsushima, a splitting
of $\varphi^* {\mathcal E}_{r,q} \to V/L^{\prime}$ as
$$
\varphi^* {\mathcal E}_{r,q}  \cong \bigoplus_{\sigma \in K} {\left({\mathcal E}_{1,q}\right)}_\sigma
$$
where ${({\mathcal E}_{1,q})}_\sigma$ is a translate of ${\mathcal E}_{1,q}$ and any two different translates being non isomorphic as holomorphic line bundles. \par
Therefore, one has an injective correspondence 
$$
{\mathcal M}: H^0(V/L,{\mathcal E}_{r,q}) \rightarrow \,\oplus_{\sigma \in {\mathbb Z}_r}H^0(V/L^{\prime},{({\mathcal E}_{1,q})_\sigma}) \,
$$
given by (picking an orthonormal basis $\{s_m\}$, $m = 0,1,\dots,q-1$)
$$
{\mathcal M}: s_m \mapsto {\rm vec} (\vartheta_{m}) :=
\left[ (\sigma \cdot \vartheta_{m})_{\sigma \in {\mathbb Z }_r}\right]  
$$
where the map $\operatorname{vec}$ arranges the translates of $\vartheta_{m} \in H^0(V/L^{\prime},{\mathcal E}_{1,q})$ into a column vector. 

We do not spell out the action of $\sigma$ in the original Matsushima picture in detail, since we shall essentially recover it anew in what follows, see Section 4.2.\par

\subsection{The noncommutative torus}
General references for the present subsection are, among others,~\cite{ConnesNCG, Pedersen, BB}.
Recall that a $C^*$-algebra ${\mathcal A}$ is a Banach $*$-algebra whose norm $|| \cdot ||$ satisfies $|| a^* a || = || a ||^2$ for all $a \in {\mathcal A}$.
The algebra $C(X)$ of continuous functions on a locally compact topological space $X$ - vanishing at infinity if $X$ is not compact - is a commutative $C^*$-algebra, with the product given by the pointwise product of functions, the involution $*$ being complex conjugation and $|| \cdot ||$ the supremum norm.
If $X$ is not compact, $C(X)$ will not possess an identity, which can nevertheless be adjoined, this corresponding to the one-point compactification of $X$. Conversely, according to the Gel'fand-Naimark theorem, any commutative $C^*$-algebra ${\mathcal A}$ can be realized as $C(\sigma({\mathcal A}))$, with $\sigma({\mathcal A})$ being the spectrum of ${\mathcal A}$, i.e. the set of maximal ideals ${\mathcal I} \subset {\mathcal A}$, equipped with the weak $*$-topology.
In $C(X)$, a maximal ideal is of the form ${\mathcal I}_x = \{ f \in C(X) | \, f(x) = 0 \}$,  where $x \in X$. Therefore, the category of locally compact spaces and proper maps is dual to the category of commutative $C^*$-algebras and $*$-homomorphisms. In particular, a standard 2-torus ${\mathbb T}^2 :=\{ (z_1,z_2) \in {\mathbb C}^2 | \, \, |z_j| = 1, j=1,2 \}$ can be traded for $C({\mathbb T}^2)$, which is in turn generated, in view of Fourier theory, by the (unitary) multiplication operators $u_j := z_j \cdot$ acting on the Hilbert space $L^2({\mathbb T}^2, {\mathfrak m})$, with ${\mathfrak m}$ the Lebesgue measure. "Deformation" of the above algebra produces what is called a noncommutative torus. It is then possible to select the "smooth" part of it, akin to the smooth functions on a manifold. This is done immediately below.\par

Let $\theta \in \mathbb{R}$. The noncommutative torus is the pre-$C^{*}$-algebra
$\mathcal{A}_{\theta}$ consisting of rapidly decaying series
\[
  a = \sum_{n,m=-\infty}^{\infty} a_{nm}u^{n}v^{m}, \qquad a_{n,m} \in \mathbb{C}
\]
where $u,v$ are unitary operators in a Hilbert space ${\mathcal H}$ satisfying the relation
\begin{equation}
  \label{eq:nc-torus-relation}
  vu = e^{2\pi i \theta}uv.
\end{equation}
We have a natural smooth structure on $\mathcal{A}_{\theta}$ given by the
noncommutative integral
\[
  \tau(a) = a_{00}, \qquad a \in \mathcal{A}_{\theta},
\]
and noncommutative derivatives
\[
  \partial_{1} (u^{n}v^{m}) = i n u^{n} v^{m},\qquad  \partial_{2} (u^{n}v^{m}) = i m u^{n}v^{m}.
\]
In the sequel we shall take $\theta \in {\mathbb Q}$, $\theta > 0$ and, ultimately, we shall deal with $\theta = q/r$, with $q$ and $r$ positive and coprime. Also, we notationally distinguish $\mathcal{A}_{\theta}$ from its $C^*$-completion $A_{\theta}$.\par
We remark from the outset that finite dimensional irreducible unitary representations of  $\mathcal{A}_{\theta}$ exist, see e.g.~\cite{ConnesNCG, Spera2015}:  
indeed set $\nu := q/r$, with $q$ and $r$ relatively prime positive integers, and
take
$u = {\rm diag}(1, e(\nu ), e(2\cdot \nu),\dots e((r-1)\cdot \nu))$
and $v$ = matrix of the shift map $e_i \to e_{i-1}$, $i=1,2,\dots r$, $e_0 = e_r$, with $(e_1,\dots, e_r)$ being the canonical basis of ${\mathbb C}^r$
and where we defined, for real $x$, $e(x) := e^{2\pi i x}$.  Then (1) is satisfied; also notice that  $u^r = v^r = 1$ (unit matrix), which entails irreducibility.
 \par
Proposition 3.3 below will show that (up to phase factors and unitary equivalence) this is indeed the typical example. In Section 4 we shall present a distributional realization of the Matsushima spaces carrying explicit noncommutative tori representations akin to the one just discussed.

\subsection{The Canonical Commutation Relations and the quantum harmonic oscillator}

In this subsection we assemble basic information on the quantum harmonic
oscillator and its relationship to the Canonical Commutation Relations and the
associated Stone-von Neumann theorem~\cite{vonNeumann}, referring to the
comprehensive survey~\cite{Rosenberg} for elucidation of their modern
ramifications.\par
Let us consider a representation of the \emph{Canonical} (or Weyl-Heisenberg)
\textit{Commutation Relations} (CCR) on a (necessarily infinite dimensional)
separable Hilbert space ${\mathcal H}$,
$$
[Q, P] = i \pmb{1}
$$
(one degree of freedom), with $Q$ and $P$ (``position'' and ``momentum''
operators, respectively) unbounded self-adjoint operators on a suitable domain.
In order to avoid problems arising from the latter issue (see
however~\cite[Section X.6]{Reed-Simon}, for amplification and further use, together
with~\cite{Dixmier}) the CCR are reformulated (Weyl) in integral form:
$$
U(a)\,V(b) = e^{i a b} \,V(b)\,U(a), \qquad a, b \in {\mathbb R}
$$
with $P$ and $Q$ becoming the infinitesimal generators of the one parameter unitary groups
$U(\cdot)$ and $V(\cdot)$, respectively.\par

The quantum harmonic oscillator Hamiltonian reads
$$
H = \frac{1}{2} (P^2+Q^2) = A^{\dagger}A +  \frac{1}{2}\pmb{1} = \frac{1}{2} (A^{\dagger}A + AA^{\dagger})
$$
in terms of \emph{annihilation} and \emph{creation} operators
$$
A = \frac{1}{\sqrt{2}}(Q+iP) \qquad A^{\dagger} = \frac{1}{\sqrt{2}} (Q-iP)
$$
subject to the commutation relation
$$
AA^{\dagger} - A^{\dagger}A = \pmb{1}.
$$
In the irreducible case the spectrum of $H$ only consists of simple eigenvalues
$\{n+ \frac{1}{2} \}_{n=0}^{\infty}$ and the $n$-{th} eigenspace ${\mathcal H}_n$ is
generated by $\phi_n = (1/\sqrt{n!}) ({A^{\dagger}})^n\phi_0$, with the ground state $\phi_0$
fulfilling $A\phi_0 = 0$. The operator $A^{\dagger}A$, namely, the Hamiltonian without
constant term (``zero-point energy'') is called the {\it number operator}.\par
In general the multiplicity of a representation of the CCR (phrased into Weyl's
integral form) is given by $k= {\dim}\,{\mathcal H}_0$: this is a version of the
Stone-von Neumann uniqueness theorem (see e.g.\ \cite{vonNeumann}).

\subsection{Gauss sums}

First of all, let us recall the celebrated \emph{Gauss sums} (see~\cite[,
Section 5.6]{HW}):
$$
S(\mu,r) := \sum_{0 \leq \ell \leq r-1} e^{2\pi i \ell^{2} \,\frac{\mu}{r}},
$$
for integers $\mu$ and $r$, the latter different from zero,
together with the well-known multiplicative formula 
$$
S(\mu q,r) S(\mu r,q) = S(\mu,rq)
$$
valid for coprime integers $r$ and $q$ and any integer $\mu$ (cf.~\cite[Theorem
64]{HW}).

Here is a quick outline of the proof. The r.h.s.\ reads
$$
\sum_{k=0}^{rq-1} e^{2\pi i \frac{\mu}{rq} \,k^2}.
$$
Now, upon exploiting the group isomorphism
$$
{\mathbb Z}_r \times {\mathbb Z}_q \, \cong \, {\mathbb Z}_{r q}
$$
stemming from the equation
$$
q \cdot [\ell]_r  + r\cdot [m]_q = [k]_{rq}
$$
which, given a residue class $[k]$ modulo $rq$, yields unique residue classes $[\ell]$ modulo $r$
and $[m]$ modulo $q$ (the converse being clear), we see that,
setting $ k = \ell q + m r$ (no dependence on representatives), the r.h.s.\ splits into the product appearing in the l.h.s. Explicitly:
$$
\frac{k^2}{rq} = \frac{(\ell q + m r)^2}{rq} = \frac{\ell^2 q}{r} + \frac {m^2 r}{q} + 2\ell m   
$$
and the last term in the r.h.s.\ exponentiates to 1. \ Notice that the problem
of finding $[k]_{rq}$ such that $[k]_{r} = [\ell]_{r}$ and $[k]_{q} = [m]_{q}$,
with given classes $[\ell]_{r}$ and $[m]_{q}$ is solved via the Chinese Remainder
Theorem: if $a$ and $b$ are integers such that $a \,r + b \,q = 1$, then
$k = qb\ell + ram$, see again~\cite[Theorem 121]{HW}.

\section{Representations of noncommutative tori and HE-vector bundles}

In this Section we reinterpret the Matsushima construction of holomorphic
HE-vector bundles over a two-dimensional torus ${\mathbb C}/ \Lambda$
---with lattice $\Lambda = \langle 1, \tau\rangle$ and $\Im \tau > 0$--- via representations of the two-dimensional noncommutative torus, see also~\cite{Spera2015}. This will be unfolded through the following series of propositions.

\begin{proposition}\label{thm:correspondence-nctorus-bundles}
	Given $\theta \in {\mathbb Q}$, $\theta > 0$, an irreducible representation of $\mathcal{A}_{\theta}$ on a finite-dimensional Hilbert space
	$\mathcal{H}$ produces a HE-vector bundle $E_{\theta} \to {\mathbb C}/ \Lambda$
	over a two-dimensional torus
    $\mathbb{C}/ \Lambda$, where $\Lambda = \langle 1, \tau\rangle$, $\Im (\tau) > 0$ with degree $\theta \dim(\mathcal{H})$ and rank
	$\dim(\mathcal{H})$.
\end{proposition}

\begin{proof}
	Let $u$, $v$ be unitary operators on a finite-dimensional Hilbert space $\mathcal{H}$
	satisfying~\eqref{eq:nc-torus-relation} (abuse of notation); then if $\gamma = n + \tau m \in \Lambda$ the function
  (theta multiplier)
  \begin{equation}
    \label{eq:theta-multiplier}
    J^{-1}_{\gamma}(z) = e^{\frac{\theta \pi}{\Im(\tau)} \left( z \overline{\gamma} + \frac{1}{2} \lvert \gamma \rvert^{2} \right)}e^{\pi i \theta nm} u^{-n} v^{-m},\qquad  z  = x + \tau y\in \mathbb{C}, \, \gamma \in \Lambda,
  \end{equation}
	satisfying
	$$	
	J^{-1}_{\gamma+\delta}(z) = J^{-1}_{\gamma}(z + \delta)J^{-1}_{\delta}(z),
$$
	defines a holomorphic vector bundle $E$  with typical fibre $\mathcal{H}$ over the torus $\mathbb{C} / \Lambda$ given as the quotient
\[
	(\mathbb{C} \times \mathcal{H}) / \sim
\]
where
\[
	(z + \gamma , {\bf v}) \sim (z, J^{-1}_{\gamma}(z) \,{\bf v}),
\]
${\bf v} \in \mathcal{H}$. Its (smooth) sections $s: \mathbb{C} \to \mathcal{H}$ (collectively denoted by $\Gamma(E)$) are then characterized by the following periodicity conditions:
	\begin{align}
		s(z+ 1)  = e^{\frac{\theta \pi}{\Im(\tau)} \left( z + \frac{1}{2}  \right)}u^{*}(s(z)),\qquad s(z+ \tau)  = e^{\frac{\theta \pi}{\Im(\tau)} \left( z\overline{\tau} + \frac{1}{2}\lvert \tau \rvert^{2}  \right)}v^{*}(s(z)).
	\end{align}
  This can be ascertained via the following computation. Write
  $$
  \tilde{s}(z) = s^j(z){\mathbf e}_j (z)
  $$
  with ${\mathbf e}_j (z)$, $j=1,2,\dots,\mathrm{dim}{\mathcal H}$ a smooth frame (Einstein's convention employed). Then
  $$
   \tilde{s}(z)  = \tilde{s}(z+\gamma) = s^j(z+\gamma){\mathbf e}_j (z+\gamma) = s^j(z+\gamma) [J^{-1}_{\gamma}(z)]^i_j {\mathbf e}_i (z) 
  $$
  whence
  $$
  s^i(z+\gamma) = [J_{\gamma}(z)]^i_j \,s^j(z),
  $$
yielding (3). \par
  On $\Gamma(E)$ we have a Hermitian structure $( \cdot |\cdot ) $ given by
	\[
		(s | s')(z) = \langle s(z), s'(z) \rangle_{\mathcal{H}} h(z), \qquad h(z) =e^{- \frac{\theta \pi}{\Im(\tau)} \lvert z \rvert^{2} }
	\]
	with Chern connection (the unique connection compatible with the Hermitian and the holomorphic structure)
	\[
		\nabla = \left( d  - \frac{\theta \pi}{\Im(\tau)} \bar{z} dz \right) \otimes \pmb{1}_{\mathcal{H}}
	\]
	having constant curvature and Chern class
    \[
      \frac{i}{2\pi}\nabla^{2} = \theta\, \omega \otimes \pmb{1}_{\mathcal{H}}, \quad c_{1}(E) = \theta \dim(\mathcal{H})\omega\,
    \]
	
	with 
	$$
	\omega = \frac{i}{2\Im(\tau)} dz\wedge d\bar{z}.
	$$
	
	Indeed, a short computation shows that, if $Q := i \nabla_{\frac{\partial}{\partial x}}$, $P :=	i \nabla_{\frac{\partial}{\partial x}}$, then
	\[
	 \frac{1}{2 \pi i}[Q, P] = \theta \,\pmb{1}_{\Gamma(E)}.
	 \]
	Moreover, it is clear that the rank of $E$ is $c_{0}(E) = \dim(\mathcal{H})$.
	
	This vector bundle will be our $E_\theta$.
\end{proof}

\begin{proposition}
	The correspondence that assigns to each representation $\pi$ of $\mathcal{A}_{\theta}$ 
	the holomorphic vector bundle ${\mathcal E}_{\pi}$ over the torus ${\mathbb C}/\Lambda$
	is functorial.
\end{proposition}

\begin{proof}
	
	Let $\pi: \mathcal{A}_{\theta} \to \mathcal{B} (\mathcal{H}_{\pi})$ and
	$\sigma: \mathcal{A}_{\theta} \to \mathcal{B}(\mathcal{H}_{\sigma})$ be two such
	representations and let $T : \mathcal{H}_{\pi} \to \mathcal{H}_{\sigma}$ be an
	intertwining unitary map. Then the map on sections
	\[
		\psi_{T}: \Gamma (\mathcal{E}_{\pi}) \to \Gamma (\mathcal{E}_{\sigma})
	\]
	given by
	\[
		(\psi_{T}(s) )(z) = T (s (z))
	\]
	is an isomorphism of $C^{\infty}(\mathbb{T}^2)$-modules. The above map
	is indeed well defined, i.e. it maps sections to sections:
	\begin{align*}
		(\psi_{T}(s) )(z + n + \tau n) & = e^{\frac{\theta \pi}{\Im(\tau)} \left( z \overline{\gamma} + \frac{1}{2} \lvert \gamma \rvert^{2} \right)}e^{\pi i \theta nm} T \pi(u)^{-n} \pi(v)^{-m} s(z)                  \\
		                               & = e^{\frac{\theta \pi}{\Im(\tau)} \left( z \overline{\gamma} + \frac{1}{2} \lvert \gamma \rvert^{2} \right)}e^{\pi i \theta nm}  \sigma(u)^{-n} \sigma(v)^{-m} T s(z)           \\
		                               & = e^{\frac{\theta \pi}{\Im(\tau)} \left( z \overline{\gamma} + \frac{1}{2} \lvert \gamma \rvert^{2} \right)}e^{\pi i \theta nm}  \sigma(u)^{-n} \sigma(v)^{-m} (\psi_{T} s)(z).
	\end{align*}
	
\end{proof}

\begin{proposition}
  Let $\pi: \mathcal{A}_{q/r}\to \mathcal{B}(\mathfrak{H})$ be an irreducible finite dimensional representation of the noncommutative torus $\mathcal{A}_{q/r}$
  where $r$ and $q$ are positive and coprime. Then, the dimension of $\mathfrak{H}$ is $r$
  and
  $\pi(u) =: {\mathfrak u}$, $\pi(v) =: {\mathfrak v}$ fulfill

  \begin{equation}
    {\mathfrak u}^{r} = \mu \pmb{1} \qquad \text{and}\qquad {\mathfrak v}^{r} = \nu \pmb{1},
  \end{equation}
  for some $\mu$, $\nu \in S^{1}$.
\end{proposition}

\begin{proof}

  Let ${\mathfrak u}$, ${\mathfrak v}$ define a finite dimensional irreducible representation of $\mathcal{A}_{q/r}$ on $\mathfrak{H}$ of
  dimension $d$. Let us write $\theta := q/r$ with $\gcd(r,q) =1$. Taking
  the determinant of $vu = e^{2 \pi i \theta}uv$ we see that $\theta d \in \mathbb{Z}$.
  Also observe that, since ${\mathfrak u}^{r}$
  and ${\mathfrak v}^{r}$ commute with the representation
  by Schur's lemma, there is a constant $\mu \in S^{1}$ such that ${\mathfrak u}^{r} = \mu \pmb{1}$;
  therefore, the minimal polynomial of $u$, call it ${\mathcal P}$,  has to divide $x^{r}- \mu$ and has degree
  at most $r$. Moreover, ${\mathcal P}$ satisfies ${\mathfrak v}{\mathcal P}({\mathfrak u}){\mathfrak v}^{*} = {\mathcal P}({\mathfrak u} e^{2 \pi i \theta}) = 0$ and
  the polynomial ${\mathcal Q}(x) = {\mathcal P}(xe^{2 \pi i \theta})$ must also satisfy ${\mathcal Q}({\mathfrak u}) = 0$; thus,
  given a
  root $\lambda$ of ${\mathcal P}$ we see that $e^{2 \pi i \theta} \lambda$ is a different root of it, whence the polynomial
  ${\widetilde {\mathcal P}}(x):= (x- \lambda)(x - e^{2\pi i /r}\lambda) \cdots (x - e^{2 \pi i (r-1)/r}\lambda)$ divides ${\mathcal P}$ and, having the same
  degree as ${\mathcal P}$, coincides with it. In particular ${\mathcal P}(x) = x^{r} - \lambda^{r}$. Let $\varphi \neq 0$ be an eigenvector
  for ${\mathfrak u}$, then $\{ {\mathfrak u}^{n}{\mathfrak v}^{m} \varphi \}_{n,m=0}^{r-1}$ generate the whole Hilbert space, and one checks that 
  ${\mathfrak u}^{n}{\mathfrak v}^{m} \varphi = \lambda^{n} e^{-2\pi i nm \theta} {\mathfrak v}^{m} \varphi$ so the
  dimension $d$ of the Hilbert space $\mathfrak{H}$ is at most $r$, and therefore equal
  to $r$.

\end{proof}

Note that if $s,t\in S^{1}$ and $s^{r} \not= 1, t^{r}\not=1$ then
\[
 {\mathfrak u}^{\prime} \equiv \pi'(u) := s {\mathfrak u} = s \pi(u), \qquad     {\mathfrak v}^{\prime} \equiv \pi'(uv) := t {\mathfrak v} = t \pi(v),
\]
defines a second irreducible representation $ \pi^{\prime}\not\cong  \pi$, since the minimal
polynomial of ${\mathfrak u}^{\prime}$ is now given by ${{\mathfrak u}^{\prime}}^{r} - s^{r}\mu \,\pmb{1}$ (notational abuse) and it is intrinsic to a representation. \qed \par
\smallskip

Recall that two noncommutative tori $A:= A_{\vartheta}$ and $B:= A_{\vartheta^{\prime}}$ are called dual, or strongly Morita equivalent  if there exists a $A - B$-bimodule $E$ such that they are each other's
endomorphism algebra, see e.g. \cite{Rieffel}. This is tantamount to require that $\vartheta$ and $\vartheta^{\prime}$ are on the same $SL(2,{\mathbb Z})$-orbit.
It turns out that all {\it rational} noncommutative tori are strongly Morita equivalent to $C({\mathbb T}^2)$, the $C^*$-algebra of continuous functions on the torus ${\mathbb T}^2$ (\cite{Rieffel-irrational}). This entails the following:

\begin{corollary}
	Since $\mathcal{A}_{\theta}$ is strongly Morita equivalent to the classical torus, its
	representations are indexed by points in $\mathbb{T}^{2}$, so they all produce, \emph{topologically},
	the same vector bundle.
\end{corollary}

\begin{proof}
Let $\pi$ denote a finite dimensional irreducible representation of the noncommutative torus $\mathcal{A}_{\theta}$ via operators $u$, $v$ as above, then any other representation on the same space is given by $u^{\prime} := \mu u, v^{\prime} := \nu v$, $\mu$, $\nu \in S^1$. Accordingly, we get another theta character $J^{\prime}$ and a corresponding holomorphic vector bundle ${\cal E}^{\prime}$, together with a naturally induced isomorphism with ${\mathcal E}$.
\end{proof}

\begin{proposition}
Let $\theta = \frac{q}{r} $, with $q$ and $r$  positive and coprime. Let $\pi$ be an irreducible  representation of $\mathcal{A}_{\theta}$ on a finite dimensional Hilbert space ${\mathcal H}$ via operators $u$, $v$, inducing as above a  HE-vector bundle $E_{\theta} \to \mathbb{C}/\Lambda$.
  Then, on the space
  $\widehat{\mathcal H}$ consisting of the $L^2$-sections of ${E }_{\theta}$ and in
  particular on the holomorphic ones (i.e.\ the {\rm Matsushima generalised
    (vector) theta functions}) we have:\par
    \smallskip
    (i) a natural $C^*$-representation of $\mathcal{A}_{\theta}$ given by operators ${\check u} := e^{iQ}$, ${\check v} := e^{iP}$; \par
    \smallskip
    (ii) a natural $C^*$-representation of $\mathcal{A}_{1/\theta}$
  given by operators ${\hat u}$, ${\hat v}$:
  \begin{equation}
    (\hat{u} s)(z) = e^{\frac{\theta\pi}{\Im(\tau)}\left( \frac{\overline{\tau}}{\theta}z - \frac{1}{2}\lvert \frac{\tau}{\theta} \rvert^{2}\right)}s\left(z - \frac{1}{\theta} \tau\right)
  \end{equation}
  \begin{equation}
    (\hat{v} s)(z) = e^{\frac{\theta\pi}{\Im(\tau)}\left(\frac{1}{\theta}z - \frac{1}{2} \lvert \frac{1}{\theta} \rvert ^{2}  \right)}s\left(z - \frac{1}{\theta}\right),
  \end{equation}
  that is, $\hat{u}, \hat{v}$ are unitary operators satisfying
  $\hat{v}\hat{u} = e^{2\pi i\frac{1}{\theta}}\hat{u}\hat{v}$. The two representations mutually commute.
\end{proposition}

\begin{proof} 

 Assertion (ii) being clear via a straightforward computation, let us elaborate on (i).
 By virtue of   Proposition 3.3, ${\rm dim} \,{\mathcal H} = r$, so we
shall denote $E_{\theta}$ also as ${\mathcal E }_{r,q}$.
Then
notice that $Q = i\nabla_{\frac{\partial}{\partial x}}, P = i\nabla_{\frac{\partial}{\partial y}}$ are symmetric and essentially
self-adjoint on $\Gamma({\mathcal E}_{r,q})$
since for all $s$, $s^{\prime} \in \Gamma({\mathcal E}_{r,q})$
$$
\int_{{\mathbb C}/\Lambda} \left[\left(\nabla_{\frac{\partial}{\partial x}} s |s^{\prime}\right) + \left( s | \nabla_{\frac{\partial}{\partial x}}s^{\prime}\right)\right] =
\int_{{\mathbb C}/\Lambda} \frac{\partial}{\partial x} (s |s^{\prime}) = 0
$$
and similarly for $\nabla_{\frac{\partial}{\partial y}}$. Essential self-adjointess ultimately
follows from Nelson's analytic vector theorem, see e.g.\ \cite{Reed-Simon}, X.6, together with Example 2.\par
Then, on the same domain, we
have
\[
  \frac{1}{2 \pi i} [Q, P] = \theta \cdot \pmb{1}_{\widehat{\mathcal H}}
\]
and we shall check below that $P$ and $Q$ and hence
$\nabla_{\frac{\partial}{\partial\bar{z}}} = -\frac{\tau}{2 \Im(\tau)i}Q  + \frac{1}{2\Im(\tau)i}P$
\emph{commute} with $\hat{u}, \hat{v}$. So in the space $\widehat{\mathcal{H}}$  we have operators satisfying
\[
  \hat{v}\hat{u} = e^{2 \pi i \frac{1}{\theta}} \hat{u}\hat{v}
\]
\[
  U(a)V(b) = e^{2 \pi i \theta ab}V(b)U(a),\qquad a,b \in \mathbb{R}
\]
where $U(a) = e^{iaP}, V(b)= e^{ibQ}$ (parallel transport operators along the fundamental directions). In particular, setting ${\check u} := V(1)$ and ${\check v} := U(1)$ (caveat) we get
$$
{\check v} {\check u} = e^{2 \pi i \theta}{\check u}{\check v}.
$$
It is then enough to show that, on $\Gamma({\mathcal E}_{r,q})$,
\[
  [Q, \hat{u}] = 0 = [Q, \hat{v}]
\]
and that the same relation holds for $P$. Let us start with the proof for $Q$.
Let $\alpha = \frac{\theta \pi}{\Im(\tau)}$. We have
\[
  -i Q\hat{v} s(z) = \frac{\partial}{\partial x} \hat{v}s(z) - \alpha z \hat{v}s(z)
\]
\[
  -i \hat{v}Q s(z) = \widehat{v}\frac{\partial s}{\partial x} (z) - \alpha \left(z- \frac{1}{\theta}\right) \hat{v}s(z)
\]
so if $(\hat{v}s)(z) = e^{\beta(z)}s(z- 1/\theta)$,
\[
  [-i Q, \hat{v}]s(z) = \left[\frac{\partial}{\partial x}, \hat{v}\right]s(z) - \frac{\alpha}{\theta}\hat{v}s(z) = \frac{\partial\beta}{\partial x} e^{\beta}s(z - 1/\theta) - \frac{\alpha}{\theta}e^{\beta}s(z - 1/\theta) = 0.
\]
Now, repeating the computation for $P$ we find:
\[
  -i P\hat{v} s(z) = \frac{\partial}{\partial y} \hat{v}s(z) - \alpha \tau z \hat{v}s(z)
\]
\[
  -i \hat{v}P s(z) = \widehat{v}\frac{\partial s}{\partial y} (z) - \alpha \tau \left(z- \frac{1}{\theta}\right) \hat{v}s(z).
\]
Thus
\[
  [-i P, \hat{v}]s(z) = \left[\frac{\partial}{\partial y}, \hat{v}\right]s(z) - \tau\frac{\alpha}{\theta}\hat{v}s(z) = \frac{\partial\beta}{\partial y} e^{\beta}s(z - 1/\theta) - \tau\frac{\alpha}{\theta}e^{\beta}s(z - 1/\theta) = 0.
\]
In the same vein, if $\hat{u}s(z) = e^{\gamma}s(z - \tau /\theta)$,
\[
  [-i Q, \hat{u}]s(z) = \left[ \frac{\partial}{\partial x}, \hat{u} \right]s(z) - \alpha \overline{\tau} \frac{1}{\theta}e^{\gamma} s(z- \tau/\theta) = \left(\frac{\partial \gamma}{\partial q} - \alpha \overline{\tau}/\theta\right)e^{\gamma}s(z - \tau/\theta) = 0
\]
\[
  [-i P, \hat{u}]s(z) = \left[ \frac{\partial}{\partial y}, \hat{u} \right]s(z) - \alpha \lvert \tau \rvert^{2} \frac{1}{\theta}e^{\gamma} s(z- \tau/\theta) = \left(\frac{\partial \gamma}{\partial y} - \alpha \lvert\tau \rvert^{2} /\theta\right)e^{\gamma}s(z - \tau/\theta) = 0
\]

yielding the conclusion.
\end{proof}

We may rephrase the previous result in the following manner. 
\begin{proposition}
Within the above setting, we have a representation of the
Heisenberg group $\mathbb{H} = \mathbb{C} \times S^{1}$ with parameter $\theta$ and a representation of the discrete Heisenberg group $\mathbb{H}_{r} = \Lambda \times S^{1}$ with parameter $r$:
\[
  W(z,t)= e^{2 \pi i t\theta} e^{i\pi \theta xy}V(x)U(y), \qquad z = x + \tau y \in \mathbb{C}
\]
\[
  \hat{w}(n,t)= t^{r} e^{i\pi \theta n_{1}n_{2}} \hat{v}^{n_{1}}\hat{u}^{n_{2}},\qquad n = n_{1} + \tau n_{2} \in \Lambda
\]
where the group structures are given, respectively, by
\[
 (z,t)(z',t') = (z + z', tt'e^{i\pi(x'y -y'x)}),  \qquad z, z' \in \mathbb{C}, \quad t,t' \in S^1
\]
and
\[
  (n,t)(n',t') = (n + n', tt'e^{i\pi \frac{1}{r}(n_{1}'n_{2} - n_{1}n_{2}')}), \qquad n, n' \in \Lambda.
\]

These representations commute:
\[
  W(z,t)\hat{w}(n,s) = \hat{w}(n,s) W(z,t)
\]
for all $z,t,n,s$.
\end{proposition}
\qed

\par
We also notice the following consequence.

\begin{proposition}
	We have $\hat{v}^{q} = \mu \,\pmb{1}_{\widehat{\mathcal H}}, \hat{u}^{q} = \nu\,\pmb{1}_{\widehat{\mathcal H}}$, where
	$\mu$, $\nu \in S^{1}$ are constants given by $u^{r} = \mu \,\pmb{1}_{\mathcal H}, v^{r}=\nu \,\pmb{1}_{\mathcal H}$.
\end{proposition}

\begin{proof}
  We find, successively
  \begin{align*}
    (\hat{v}^{q} s)(z) & = e^{\frac{\theta\pi}{\Im(\tau)}\left(\frac{1}{\theta}\left( z + (z - 1/\theta) - \cdots - (z - (q-1)/ \theta) \right) - \frac{q}{2} \lvert \frac{1}{\theta} \rvert ^{2}  \right)}s\left(z - \frac{q}{\theta}\right) \\
                       & = e^{\frac{\theta\pi}{\Im(\tau)}\left(\frac{1}{\theta}\left( qz - \frac{{q(q-1)}}{2\theta} \right) - \frac{q}{2} \lvert \frac{1}{\theta} \rvert ^{2}  \right)}s\left(z - r\right)                                    \\
                       & = e^{\frac{\theta\pi}{\Im(\tau)}\left(rz - \frac{{r^{2}(q-1)}}{2q} - \frac{r^{2}}{2q}\right)}s\left(z - r\right)\\
                       & = e^{\frac{\theta\pi}{\Im(\tau)}\left(rz - \frac{r^{2}}{2}\right)}s\left(z - r \right)\\
                       & = e^{\frac{\theta\pi}{\Im(\tau)}\left(rz - \frac{r^{2}}{2}\right)}e^{\frac{\theta\pi}{\Im(\tau)}\left(-rz + \frac{r^{2}}{2}\right)} (u^{r})(s(z))\\
                       & = \mu \, s(z),
  \end{align*}
  where we used
  (3)  in the second to last equality. Similarly,
  \begin{align*}
    (\hat{u}^{q} s)(z) & = e^{\frac{\theta\pi}{\Im(\tau)}\left(\frac{\overline{\tau}}{\theta}\left( z + (z - \tau/\theta) - \cdots - (z - (q-1)\tau/ \theta) \right) - \frac{q}{2} \lvert \frac{\tau}{\theta} \rvert ^{2}  \right)}s\left(z - \frac{\tau q}{\theta}\right) \\
                       & = e^{\frac{\theta\pi}{\Im(\tau)}\left(\frac{\overline{\tau}}{\theta}\left( qz - \frac{\tau q(q-1)}{2\theta} \right) - \frac{q}{2} \lvert \frac{\tau}{\theta} \rvert ^{2}  \right)}s\left(z - \tau r\right)                                    \\
                       & = e^{\frac{\theta\pi}{\Im(\tau)}\left(r\overline{\tau}z - \frac{r^{2}(q-1) \lvert\tau\rvert^{2}}{2q} - \frac{r^{2}\lvert\tau\rvert^{2}}{2q}\right)}s\left(z - \tau r \right)\\
                       & = e^{\frac{\theta\pi}{\Im(\tau)}\left(r \overline{\tau} z - \frac{r^{2}\lvert \tau \rvert^{2}}{2}\right)}s\left(z - \tau r \right)\\
                       & = e^{\frac{\theta\pi}{\Im(\tau)}\left(r \overline{\tau} z - \frac{r^{2}\lvert \tau \rvert^{2}}{2}\right)} e^{\frac{\theta\pi}{\Im(\tau)}\left(-r \overline{\tau} z + \frac{r^{2}\lvert \tau \rvert^{2}}{2}\right)} (v^{r})(s(z))\\
                       & = \nu \, s(z).
  \end{align*}
\end{proof}

Note that, in accordance with the Stone-von Neumann theorem,  the representation $W$ is not irreducible: indeed, by Riemann-Roch, its multiplicity is precisely $\theta\,\dim(\mathcal{H}) = (q/r) \cdot r = q$, also cf.~\cite{Spera1986, Spera2015}.
 
\begin{corollary}
   Each $k$-level theta line bundle over the two-dimensional torus produces an
   irreducible finite dimensional representation of $\mathcal{A}_{\theta}$ with $\theta = 1/k$.
\end{corollary}

In view of the preceding discussion (Subsection 2.4) on the harmonic oscillator we have (with obvious and inessential notational changes), the following result.

\begin{proposition}
  Let $\Delta = \left(\nabla_{\partial_{\bar{z}}}\right)^*\nabla_{\partial_{\bar{z}}} \equiv A^*A $  be the Laplacian  on $E_{\theta}$ (the ``number operator''). Then its spectrum only consists of eigenvalues,
  whose eigenspaces are finite-dimensional with the same dimension $q$ and each carrying a representation of ${\mathcal A}_{1/\theta} = {\mathcal A}_{r/q}$ (with generators ${\hat u}$, ${\hat v}$).
\end{proposition}

\begin{remark}
  The above proposition, together with the preceding developments, reformulate
  and possibly improve (in the classical case) the celebrated results
  establishing the action of finite Heisenberg groups on spaces of theta
  functions (viewed as holomorphic sections of line bundles on complex tori),
  see e.g.\ \cite{Mumford} (esp.\ vol. III) and~\cite{Polishchuck} (cf.\ in particular the remark in Section 3.1).
\end{remark}

\begin{proposition}
  (Bimodule structure). The HE-vector bundle
  ${\mathcal E}_{r,q} \to {\mathbb C}/\Lambda$ (actually, its space of smooth sections
  $\Gamma({\mathcal E}_{r,q})$) comes equipped with a
  ${\mathcal A}_{\theta}-{\mathcal A}_{-1/\theta}$ bimodule structure, where $\mathcal{A}_{\theta}$ acts
  on the left by ${\check u}$ and ${\check v}$, and $\mathcal{A}_{1/\theta}$ acts on the right by
  ${\hat u}$ and ${\hat v}$.
\end{proposition}

\begin{proof}
  This is clear in view of the preceding discussion.  The minus sign comes from regarding ${\hat u}$ and ${\hat v}$ as acting on the right.
\end{proof}

\begin{remark}
Regarding the above mentioned strong Morita equivalence between $\mathcal{A}_{\theta}$  and ${\mathcal A}_{-1/\theta}$, it would be interesting to explicitly compare the two algebra-valued Hermitian structures involved, see e.g \cite{Rieffel}.
\end{remark}

\section{Gauss sums, vector theta functions and the FMN-transform}

The above construction can be interpreted in terms of the so-called Fourier Mukai-Nahm (FMN*) transform (plus dualization) as in~\cite{Spera2015} (see in particular Section 4.3). For background on the FMN-transform see e.g.\  ---in addition to the original sources~\cite{Mukai,Nahm}--- the article~\cite{Tejero} and the textbook~\cite{Bruzzo&al}.\par
Specifically, in view of Proposition 3.3, an irreducible representation 
$\pi^{\prime}$ of ${\mathcal A}_{{1}/{\theta}}$ on a finite dimensional Hilbert space ${\mathcal H}^{\prime}$
yields in turn, \emph{\'a la} Matsushima, a HE-vector bundle $E_{1/\theta} \to \mathbb{C}/\Lambda$ with rank $q = {\dim}\,H^0(E_{\theta}) = {\dim} \,\mathcal{H}^{\prime}$ and degree $r$ (FMN*- dual to $E_{\theta} \to \mathbb{C}/\Lambda$) equipped with a Chern connection
$\nabla^{\prime}$ having  constant curvature and Chern class given respectively by
\[
  \frac{i}{2\pi}{\nabla^{\prime}}^{2} =  (1/\theta)\, \omega \,{\mathbf 1}_{\mathcal{H}^{\prime}}, \quad c_{1}(E_{\frac{1}{\theta}}) = (1/\theta) \dim(\mathcal{H}^{\prime})\, \omega.
\] 
The above connection can be also readily computed via noncommutative geometric tools as in~\cite{Spera2015}.
In general, the moduli dependence is governed by Proposition 3.6.\par
In the following sections we shall reformulate the Matsushima approach by a further enhancement of a noncommutative torus perspective and by enforcing FMN* from the outset via a matrix portrait  and by building upon the Coherent State Transform of~\cite{FMN}.

\subsection{$\delta$-description of vector theta functions and Gauss sums}

Let us consider coprime positive integers $r$ and $q$, set, for $x \in S^1$,
$$
\delta_{\ell}^{(q)} (x) := \delta [(x - \ell/r)\, q], \qquad\qquad \ell =0,1,\dots,r-1.
$$
If $q=1$ we simply write $\delta_{\ell}$ instead of $\delta_{\ell}^{(q)}$.
From the (distributional) Fourier expansion (involving a $q$-covering  $S^1 \to S^1$ and,
dually, the subgroup $q {\mathbb Z} \subset {\mathbb Z}$)
$$
\delta_{0}^{(q)} (x) = \delta (q x) = \sum_{n \in  {\mathbb Z}} e^{2\pi i n q x} \equiv {\mathcal F} \delta^{(q)}  \equiv \theta_0^{(q)} 
$$
one gets
$$
\delta_{\ell}^{(q)} (x) = \sum_{0 \leq \ell^{\prime}\leq r-1} e^{-2\pi i \ell^{\prime} \ell \,\frac{q}{r}} \sum_{n \in  {\mathbb Z}} e^{2\pi i (\ell^{\prime} + r n) q x} =  \sum_{0 \leq \ell^{\prime}\leq r-1} e^{-2\pi i \ell^{\prime} \ell \,\frac{q}{r}}  \theta_{\ell^{\prime}}^{(q)} \equiv a^{\ell^{\prime}}_{\ell} \theta_{\ell^{\prime}}^{(q)}
$$
via the introduction of the (invertible) $r \times r$ matrix (cf.~\cite{FMN})
$$
A :=  \left(a^{\ell^{\prime}}_{\ell}\right) = \left(e^{-2\pi i \ell^{\prime} \ell \,\frac{q}{r}}\right)
$$
(Einstein's convention is employed) relating the $\delta$ and (boundary) theta descriptions; thus
$$
\Tr A = \sum_{0 \leq \ell\leq r-1} e^{-2\pi i \ell^{2} \,\frac{q}{r}} = \overline{S(q, r)}
$$
i.e.\ a Gauss sum.
Similarly (obvious notation, with $y \in S^1$), one has 
$$
\delta_{m}^{(r)} (y) := \delta [(y - m/q)\, r] = \sum_{0 \leq m^{\prime}\leq q -1} e^{-2\pi i m^{\prime}m \,\frac{r}{q}} \sum_{n \in  {\mathbb Z}} e^{2\pi i (m^{\prime} + q n) r y} \equiv b^{m^{\prime}}_{m} \theta_{m^{\prime}}^{(r)}
$$
with a corresponding matrix
$$
B :=  \left(b^{m^{\prime}}_{m}\right) = \left(e^{-2\pi i m^{\prime} m \,\frac{r}{q}}\right)
$$
with
$$
\Tr B = \sum_{0 \leq m \leq q -1} e^{-2\pi i m^{2} \,\frac{r}{q}} = \overline{S(r, q)}.
$$
Then consider the tensor product distributions
$$
\delta_{\ell}^{(q)} (x) \delta_{m}^{(r)}(y), \qquad x, y \in S^1.
$$
Upon fixing $m \in \{ 0,1,\dots, q-1\}$, one has an obvious $r$-component \emph{column} vector, representing a model for the $m$-th
Matsushima holomorphic section for the vector bundle ${\mathcal E}_{r,q} \to V/L$. More precisely, we have (with a natural abridged notation), upon suitably reinterpreting Matsushima's construction (\cite{Matsushima}, Section 8, and our previous discussion on the {\it CST} in Section 2.1):
$$
{\vec{\delta}}_{\cdot,m} := (\delta_{0,m}, \delta_{1,m}, \cdots \delta_{r-1,m})^T \leftrightarrow \delta_{0,m}.
$$

The $q$ columns thus obtained yield a basis for a $q$-dimensional Hilbert space ${\mathcal H}^q \cong H^0(V/L,{\mathcal E}_{r,q})$.

Similarly,
fixing $\ell \in \{ 0,1,\dots, r-1\}$, we get a \emph{row} vector, giving rise to a model for the $\ell$-th
holomorphic section of the (FMN$^*$) dual vector bundle ${\mathcal E}_{q,r} \to V/L$,
namely:
$$
{\vec{\delta}}_{\ell,\cdot} := (\delta_{\ell,0}, \delta_{\ell,1}, \cdots \delta_{\ell,q-1}) \leftrightarrow \delta_{\ell,0}.
$$
and the ensuing $r$ rows yield a basis for an $r$-dimensional Hilbert space ${\mathcal H}^r \cong H^0(V/L,{\mathcal E}_{q,r})$.

\par
Also, in view of the previous considerations (Section 2.5), we can naturally establish a bijective correspondence 
$$
\delta_{\ell}^{(q)} (x) \delta_{m}^{(r)}(y) \, \leftrightarrow \, \delta_{k} (z) = \delta (z - k/rq), \qquad z \in S^1
$$
($k \in \{ 0,1,\dots, rq -1\}$), with the $rq$-level thetas, viewed as holomorphic sections of
the complex line bundle ${\mathcal E}_{1,rq} \to V/L^{\prime}$.  

Therefore, one finds a third matrix
$$
C :=  (c_k^{k^{\prime}} = e^{-2\pi i k^{\prime} k \,\frac{1}{rq}})
$$
with 
$$
\Tr C = \sum_{0 \leq k \leq rq -1} e^{-2\pi i k^{2} \,\frac{1}{rq}} = \overline{S(1, rq)}.
$$

The above matrix is related to the former ones in the following way.
 Let us consider the following $rq$-dimensional Hilbert spaces: ${\mathcal H}^{rq}$, generated by the
 orthonormal basis $\delta_{\ell}^{(q)}(x)\delta_{m}^{(r)}(y)$, and ${\mathbb H}^{rq}$, generated by 
 the orthonormal basis $\delta_{k}(z)$; we have then a natural unitary transformation
 ${\mathcal T}: {\mathcal H}^{rq} \rightarrow {\mathbb H}^{rq}$ 
 $$
 {\mathcal T} (\delta_{\ell}^{(q)}\delta_{m}^{(r)}) := \delta_{k}
 $$
whereby
$$
{\mathcal T} (\theta_{\ell}^{(q)}\theta_{m}^{(r)}) = \theta_{k}
$$
 as well (shorthand notation), this easily leading to
$$
C = {\mathcal T} \, (A \otimes B) \,{\mathcal T}^{-1}.
$$

Therefore, from 
$$
 \Tr \,C  = \Tr \, [{\mathcal T}\, (A \otimes B) \,{\mathcal T}^{-1}] =  \Tr\,(A \otimes B) = \Tr \,A \cdot \Tr \,B
$$
we get a special case of the above multiplicative formula for Gauss sums with
$\mu =1$. \par 
Actually, the general formula is also obtained via the same technique, after introducing
from the outset another $\mu$-covering $S^1 \to S^1$, resulting in an extra factor $\mu$ in the numerators of all arguments of the exponentials. \par 
This may be viewed as a sort of \emph{categorification} of Gauss sums in the sense that, as numerical objects, they come from the multiplicativity of tensor product traces.\par
A variant of the above procedure consists in exploiting the (algebra) isomorphism $M_{rq} ({\mathbb C}) \cong 
M_{r} ({\mathbb C}) \otimes M_{q} ({\mathbb C})$ via the elementary matrix bases $E_{ij}$ (i.e.\ the matrices
whose $(i,j)$-entry is $1$ and all others are zero):
$$
E_{kk^{\prime}} \leftrightarrow E_{\ell\ell^{\prime}} \otimes E_{mm^{\prime}}
$$
\noindent
(abridged notation: $k$, $k^{\prime}$ and so on are taken modulo the size of the respective matrices).

\begin{remark}
A few words about the \emph{heuristics} behind the above discussion are maybe in order:
upon \emph{formally} multiplying the deltas labelled by $\ell$ and $m$ after taking the same argument $x=y$ (this is an ill-defined object!),
one has, for the product of their Fourier series, after an obvious index relabelling, the (meaningless) expression
\begin{equation}
  \sum_{n,N \in {\mathbb Z}} e^{2\pi i [\ell q + m r + rq N] x}
\end{equation}
which, upon \emph{discarding} the sum in $n$, yields the distributional Fourier series expressing
 $\delta (z - k/rq)$ ---after changing $x$ to $z$--- with $[k]_{rq}$ obeying the above equation.
\end{remark}

\subsection{Noncommutative torus aspects of the $\delta$-formulation}
Set (abridged notation) $\delta_{\ell m} := \delta_{\ell}\delta_{m}$ and
define, in ${\mathcal H}^{rq}$, for $\mu$, $\nu$, $\widetilde{\mu}$,  $\widetilde{\nu} \in S^1$
$$
{\mathfrak U}\, \delta_{\ell m} := \mu \delta_{\ell, m-1}, \qquad {\mathfrak V} \, \delta_{\ell m} := \nu e^{-2\pi i\,\frac{m}{q}} \delta_{\ell m}
$$
and
$$
{\widetilde {\mathfrak U}} \, \delta_{\ell m} := \widetilde{\mu} \delta_{\ell-1,m}, \qquad {\widetilde {\mathfrak V}} \, \delta_{\ell m} := \widetilde{\nu} e^{-2\pi i\,\frac{\ell}{r}} \delta_{\ell m}
$$
(cyclic ordering understood: for instance, ${\mathfrak U} \,\delta_{\ell,0} := \mu \delta_{\ell,q}$ et cetera).
One has, upon restriction to the spaces indicated, ${\mathfrak U}^q   = \mu^q 1_{{\mathcal H}^q}$,  ${\mathfrak V}^q = \nu^q 1_{{\mathcal H}^q}$, ${\widetilde {\mathfrak U}}^r =  {\widetilde{\mu}}^r 1_{{\mathcal H}^r}$, ${\widetilde {\mathfrak V}}^r = {\widetilde{\nu}}^r 1_{{\mathcal H}^r}$ and subsequently

$$
{\mathfrak U}{\mathfrak V} =  e^{-2\pi i\,\frac{1}{q}} {\mathfrak V}{\mathfrak U}, \qquad \qquad {\widetilde {\mathfrak U}} {\widetilde {\mathfrak V}} = e^{-2\pi i\,\frac{1}{r}}  {\widetilde {\mathfrak V}}{\widetilde {\mathfrak U}}. 
$$
Then define
$$
{\mathbb U} \delta_{\ell m} := \mu \delta_{\ell, m-1}  = {\mathfrak U} \, \delta_{\ell m}, \qquad {\mathbb V}\delta_{\ell m} := \nu^r e^{-2\pi i\,\frac{r}{q}\, m} \delta_{\ell m} = {\mathfrak V}^r \, \delta_{\ell m},
$$
yielding
$$
{\mathbb U} \,{\mathbb V} = e^{-2\pi i\,\frac{r}{q}} \, {\mathbb V} \,  {\mathbb U} 
$$
that is, a representation of ${\mathcal A}_{1/\theta}$, $\theta = q/r$ (cyclic ordering again understood), with ${\mathbb U}^q = \mu^q 1_{{\mathcal H}^q}$, ${\mathbb V}^q =  \nu^{rq} 1_{{\mathcal H}^q}$. Similarly, upon setting
$$
{\widetilde{\mathbb U}} \delta_{\ell m} := {\widetilde{\mu}}\delta_{\ell-1, m}  = {\widetilde {\mathfrak U}} \, \delta_{\ell m}, \qquad {\widetilde{\mathbb V}}\delta_{\ell m} := {\widetilde{\nu}}^q e^{-2\pi i\,\frac{q}{r}\, \ell} \delta_{\ell m} =  {\widetilde {\mathfrak V}}^q \, \delta_{\ell m},
$$
we get
$$
{\widetilde{\mathbb U}} \,{\widetilde{\mathbb V}} = e^{-2\pi i\,\frac{q}{r}} \, {\widetilde{\mathbb V}} \,{\widetilde{\mathbb U}}
$$
(a representation of ${\mathcal A}_{\theta}$), together with ${\widetilde{\mathbb U}}^r =  {\widetilde{\mu}}^r 1_{{\mathcal H}^r}$,  ${\widetilde{\mathbb V}}^r =  {\widetilde{\nu}}^{rq}1_{{\mathcal H}^r}$.
These two representations mutually commute (since they do not mix first and second subscripts)  
and  they are exchanged upon application of the FMN$^*$-transform.\par
\smallskip
\noindent
The action of the various operators involved can be cast in a more compact way as follows: in ${\mathcal H}^q$
one has
$$
{\mathbb U}\, \vec{\delta}_{\cdot,m} = \mu \,\vec{\delta}_{\cdot,m-1} \qquad
{\mathbb V}\, \vec{\delta}_{\cdot,m} = \nu^r e^{-2\pi i \frac{r}{q} m} \,\vec{\delta}_{\cdot,m},
\quad m=0,1,\dots,q-1
$$
with the ``tilded'' operators acting as the identity:
$$
{\widetilde{\mathbb U}}\, \vec{\delta}_{\cdot,m} = \vec{\delta}_{\cdot,m} \qquad
{\widetilde{\mathbb V}}\, \vec{\delta}_{\cdot,m} = \vec{\delta}_{\cdot,m},
\quad m=0,1,\dots,q-1.
$$
A similar portrait, \emph{mutatis mutandis}, holds in ${\mathcal H}^r$. Summarizing, we have
\begin{proposition}
The above operators ${\mathbb U}$, ${\mathbb V}$ realize a representation of ${\mathcal A}_{r/q}$ unitarily equivalent to a representation induced by ${\hat u}$, 
  ${\hat v}$ in Proposition 3.4 above, after restriction of the latter to the space of holomorphic sections of ${\mathcal E}_{r,q}$. An analogous statement is true for the tilded operators and the FMN*- transformed bundle ${\mathcal E}_{q,r}$. 
 \end{proposition}

  \begin{remark}
  Geometrically, the above ``toric'' families of representations correspond to
  tensoring the initial holomorphic bundle $E_{\theta} \to {\mathbb C}/ \Lambda$ with the
  flat line bundle ${\mathcal P}_{\xi} \to {\mathbb C}/ \Lambda$, the restriction of the
  Poincar\'e bundle to ${\mathbb C}/ \Lambda \times \{\xi \} \cong {\mathbb C}/ \Lambda$, where
  $\xi = (\mu, \nu)$ and similarly for $E_{1/\theta} \to {\mathbb C}/ \Lambda$. Also notice that
  the torus also classifies holonomies of the different Chern  connections,
  see also~\cite{Spera2015}.
\end{remark}

We recover the standard Matsushima correspondence involving the holomorphic vector bundle ${\mathcal E}_{r,q} \to V/L$ and the $q$-level theta line bundle ${\mathcal E}_{1,q} \to V/L^{\prime}$ via the Coherent State Transform $CST$ through the following steps (obvious abridged notation), also setting
 $\mu = {\widetilde{\mu}} = {{\nu}} = {\widetilde{\nu}} = 1$  for simplicity. Indeed:
$$
\vartheta_{j,m} = CST (\delta_{j,m}) = CST ({\widetilde U}^j \delta_{0,m}) = CST ({\widetilde{\mathbb U}}^j \delta_{0,m}), \qquad j=0,1,\dots, r-1 
$$
whence
 $$
{\vec{\vartheta}_{\cdot, m}} := (\vartheta_{j,m})_{j=0,\dots r-1}^T =: CST \,(\vec{\delta}_{\cdot,m})   \leftrightarrow 
\vartheta_{0,m}.
$$
Actually, in this manner we have defined a \emph{vector} version of the Coherent State Transform:
$$
{\mathcal {CST}} =  {\mathcal M}^{-1} \circ {\rm vec} \circ CST \circ {\mathcal F}
$$
mapping
$
\delta_{0,m} \mapsto  s_m
$
and spelled out as follows
$$
\delta_{0,m} {}^{{\mathcal F}\atop\mapsto} \theta_{0,m} (x) {}^{{CST}\atop\mapsto} 
\vartheta_{0,m}(z) {}^{{\rm vec}\atop\mapsto}{\vec{\vartheta}}_{\cdot,m}(z) \!\! {}^{{\quad{\mathcal M}^{-1}}\atop\mapsto} s_m.
$$
The notation ${\mathcal M}^{-1}$ is justified since ${\mathcal M}$ is injective and ${\rm Im (vec)} \subset {\rm Im}\,{\mathcal M}$.
We may recap the previous discussion via the isomorphism 
$$
H^0({V/L, \mathcal E}_{r,q})  \otimes H^0({V/L,\mathcal E}_{q,r}) \cong H^0(V/L^{\prime},{\mathcal E}_{1,rq}) 
$$

induced by the correspondence
$$
s_m \otimes s_{\ell}  \leftrightarrow s_k
$$
where again
$$
[k]_{rq} = q \, [l]_{r} + r \, [m]_{q}
$$
and  by further noticing the following ``categorical'' result:

\begin{proposition}
  Under the above assumptions, we have
  \begin{equation}
    \mathcal{A}_{1/rq} \cong \mathcal{A}_{q/r}  \otimes \mathcal{A}_{r/q}.
  \end{equation}
\end{proposition}

\begin{proof}
  First observe that noncommutative tori are nuclear $C^*$-algebras, so their
  $C^*$-tensor product appearing in the r.h.s.\ is uniquely determined (\cite{BB}, IV.3.5.3, p. 392). Then, starting from
  $$
  u\,v = e^{-2\pi i \frac{r}{q}} \, v\, u
  $$
  and
  $$
  {\widetilde u} \,{\widetilde v} = e^{-2\pi i \frac{q}{r}}\, {\widetilde v} \,{\widetilde u}
  $$
  (all tilded operators commuting with untilded ones),
  define (same notation as before: $k = \ell q + m r$ and so on):
  $$
  {\mathcal U}^k := u^m {\widetilde u}^{\ell}, \qquad {\mathcal V}^{k^{\prime}} := v^{m^{\prime}} {\widetilde v}^{\ell^{\prime}}.
  $$
  A straightforward computation then yields:
  $$
  {\mathcal U}^k\,{\mathcal V}^{k^{\prime}} = e^{-2\pi i \frac{k k^{\prime}}{\!rq}} \, {\mathcal V}^{k^{\prime}}{\mathcal U}^k.
  $$
  The above reasoning is clearly invertible, achieving the sought for conclusion.
\end{proof}

\begin{remark}
  Upon further requiring that $u^q = v^q =1$ and
  ${\widetilde u} ^r = {\widetilde v}^r =1$, we get
  ${\mathcal U}^{rq} = {\mathcal V}^{rq} =1$ (with $1$ the identity in the
  respective algebras).\par
  Also notice that, at the vector bundle level, this reflects an operation
  (denoted by $\star$)
  \begin{equation}
    {\mathcal E}_{r,q} \star {\mathcal E}_{q,r} = {\mathcal E}_{1, rq}.
  \end{equation}
  casting light on FMN-duality via a ``Gauss'' perspective.
\end{remark}

\section{Conclusions and outlook}

In this paper we used complex algebraic-geometric and noncommutative geometric
techniques in order to understand and possibly enhance, at least in the
classical case (i.e.\ on the complex field) the  relationship between
theta functions and Heisenberg groups. Our research is strongly motivated by the
Quantum Hall Effect as well:  the symmetry $\theta \leftrightarrow 1/\theta$ discussed in the paper
may ultimately lead to an explanation of the duality occurring between Hofstadter's and Harper's regimes, see
e.g.~\cite{Denittis}.  Also, our results may help in providing a clear-cut mathematical
formulation of the important \emph{Laughlin gauge principle} for a toral
configuration, see \ e.g.\ the comprehensive review~\cite{Hatsugai}. Finally,
instances of the vector bundles dealt with in the present paper also appear in
the works~\cite{DN-L, Denittis}, devoted to a far reaching generalization of the
TKNN equations. These questions will be possibly tackled elsewhere.

\subsubsection*{Acknowledgements}

We are grateful to G. De Nittis for useful discussions and to the Referees for their careful reading and critical remarks.\par
\noindent
M. Sandoval's research is supported by the grant
\texttt{CONICYT-PFCHA Doctorado Nacional 2018--21181868}.\par
\noindent
M. Spera (member of INDAM-GNSAGA) is supported by local (D1) UCSC funds.

\end{document}